\documentclass[12pt,reqno]{amsart}
\usepackage{amsmath,bm}
\usepackage{amsfonts}
\usepackage{amssymb}
\usepackage{amsthm}
\usepackage{amscd}
\usepackage{a4wide}

\usepackage{amsmath}
\usepackage{amssymb}
\usepackage{amsthm}
\usepackage{graphics}
\usepackage{eucal}
\usepackage{mathrsfs}

\usepackage{color}
\usepackage[pdftex,bookmarks,colorlinks,breaklinks]{hyperref}  

\hypersetup{linkcolor=red,citecolor=blue,filecolor=dullmagenta,urlcolor=darkblue} 

\newtheorem {theorem}    {Theorem}[section]
\newtheorem {lemma}      [theorem]    {Lemma}

\newtheorem {corollary}  [theorem]    {Corollary}
\newtheorem {proposition}[theorem]    {Proposition}

\theoremstyle{definition}
\newtheorem {definition} [theorem]    {Definition}

\newcounter{AbcT}

\numberwithin{equation}{section}

\newcommand{\A}{\mathcal A}

\newcommand {\N} {{\mathbb N}}
\newcommand {\R} {{\mathbb R}}
\newcommand{\G}{{\mathbb G}}

\newcommand {\Z} {{\mathbb Z}}

\newcommand{\IGNORE}[1]{}

\renewcommand{\liminf}{\varliminf}
\renewcommand{\limsup}{\varlimsup}

\DeclareMathOperator{\Hom}{Hom} \DeclareMathOperator{\SL}{SL}

\DeclareMathOperator{\GL}{GL}

\DeclareMathOperator{\Mat}{M}

\newcommand{\bsl}{\backslash}
\newcommand{\Gm}{\mathbf{G}_m}

\newcommand\rank{\operatorname{rank}}

\newcommand{\fs}{\mathbb{F}}

\newcommand{\diag}{\operatorname{diag}}

\newcommand{\bbM}{\mathbb{M}}

\newcommand{\bbR}{\mathbb{R}}
\newcommand{\bbH}{\mathbb{H}}
\newcommand{\bbB}{\mathbb{B}}

\newcommand{\bbT}{\mathbb{T}}
\newcommand{\bbA}{\mathbb{A}}
\newcommand{\bbG}{\mathbb{G}}

\newcommand{\bq}{\mathbf{q}}
\newcommand{\bp}{\mathbf{p}}
\newcommand{\cH}{\mathcal{H}}
\newcommand{\bv}{\mathbf{v}}

\newcommand{\scH}{\mathfrak{H}}

\newcommand{\scF}{\mathfrak{F}}
\newcommand{\fmi}{f^{-1}_{m}}
\newcommand{\fni}{f^{-1}_{n}}

\newcommand{\DL}{\operatorname{UDL}}
\newcommand{\ED}{\operatorname{ED}}

\newcommand{\Aut}{\operatorname{Aut}}

\newcommand{\ud}{\,\mathrm{d}}
\newcommand{\scHF}{\mathfrak{H}^{\mathfrak{F}}} 

\newcommand{\Supp}{\operatorname{Supp}}
\newcommand{\aut}{\operatorname{aut}}

\newcommand{\volume}{\operatorname{volume}}

\newcommand{\dist}{\operatorname{dist}}

\newcommand{\bfk}{\mathbf{k}}
\newcommand{\bt}{\mathbf{t}}
\newcommand{\mn}{m \times n}
\newcommand{\Imn}{I_{m \times n}}

\begin{document}
\title[Ultrametric Logarithm Laws, II] {Ultrametric Logarithm Laws, II}

\begin{abstract}
We prove positive characteristic versions of the logarithm laws of Sullivan and Kleinbock-Margulis and obtain related results in Metric Diophantine Approximation.  
 
\end{abstract}
\subjclass[2010]{11J83, 11K60, 37D40, 37A17, 22E40} \keywords{
Logarithm laws, Diophantine approximation, positive
characteristic}

\author{J. S. Athreya}
\thanks{Athreya was partially supported by an NSF postdoctoral fellowship, DMS 0603636. }
\address{Department of Mathematics, University of Illinois.
1409 W. Green Street, Urbana, IL 61801.}
\author{Anish Ghosh}
\thanks {Ghosh was partially supported by an EPSRC grant.}
\address{School of Mathematics, University of East Anglia, Norwich, NR4 7TJ UK}
\author{Amritanshu Prasad}
\address{The Institute of Mathematical Sciences Taramani, Chennai 600 113, India}

\maketitle
\tableofcontents

\section{Introduction}

In the landmark work \cite{Sullivan}, Sullivan established the following important, and by now folklore theorem.  Let $V = \mathbb{H}^{d+1}/\Gamma$ be a hyperbolic manifold where $\Gamma$ is a discrete subgroup of hyperbolic isometries which is not co-compact, and let $\dist v(t)$ denote the distance from a fixed point in $V$ of the point achieved after traveling a time $t$ along the geodesic with initial direction $v$.

\begin{theorem}\label{K-S-theorem}[Khintchine-Sullivan theorem]
For all $x \in V$, and almost every $v \in T_x V$,
$$ \limsup_{t \to \infty} \frac{\dist v(t)}{\log t}  = 1/d.$$

\end{theorem}

\noindent In fact, this is a special case of more general results proved in \cite{Sullivan}. The reason the above is referred to as the \emph{Khintchine-Sullivan} theorem is that it has an intimate connection with Khintchine's theorem in Diophantine approximation.  Let a non-increasing continuous function $\psi$ from $\R_{+} \to \R_{+}$ be given and let $\mathcal{W}_{1 \times 1}(\psi, \R)$ be the subset of real numbers $x$ for which there exist infinitely many $q \in \mathbb{Z}$ such that 
\begin{equation}\label{psi1} 
|p + xq| < \psi(|q|)
\end{equation}

\noindent for some $p \in \mathbb{Z}$.  Khintchine's theorem characterizes the size of $\mathcal{W}_{1 \times 1}(\psi, \R)$ in terms of Lebesgue measure.

\begin{theorem}\label{KG}[Khintchine's theorem]
$\mathcal{W}_{1 \times 1}(\psi, \R)$ has zero or full measure according as 
\begin{equation}\label{grosum}
\sum_{x = 1}^{\infty} \psi(x)
\end{equation}
\noindent converges or diverges. 
\end{theorem}

\noindent Sullivan's result were generalized by Kleinbock and Margulis in their important paper \cite{KleinMarg}. They generalized Theorem \ref{K-S-theorem} to locally symmetric spaces of non-compact type and Theorem \ref{KG} to a multidimensional setting as well as to \emph{multiplicative} settings. These apparently disparate results are both manifestations of the \emph{shrinking target problem} (\cite{Kleinbock-Chernov}, \cite{Athreya-survey}) for actions of diagonal one-parameter subgroups on homogeneous spaces of Lie groups. Briefly, both the geodesic flow on the unit tangent bundle of a non-compact locally symmetric space as well as approximation of real numbers by rational numbers can be modelled by cusp excursions of such one-parameter subgroups on the non-compact space $G/\Gamma$ where $G$ is a semisimple Lie group and $\Gamma$ is a lattice in $G$.\\

 It is natural to investigate analogues of these results in more general settings. In this paper, we are concerned with analogues of Theorems \ref{K-S-theorem} and \ref{KG} over local fields of positive characteristic. Both these problems have been studied extensively in recent times (more on this in \S's \ref{section-diophantine} and \ref{section-geodesic} respectively).

\subsection{Notation and set-up}  We will use Vinogradov notation throughout the paper. Thus $a \ll b$ means $a \leq C b$ for a positive constant $C$ while $a \asymp b$ means $a \ll b$ and $b \ll a$. Let $\fs$ denote the finite field with $s$ elements, $s$ a power of a prime $p$. Let $\bfk
= \fs(X)$ be the ring of rational functions, $Z = \fs[X]$ the
ring of polynomials, $O = \fs[[X^{-1}]]$ be the ring of
formal series in $X^{-1}$ and $k = \fs((X^{-1}))$ denote the field
of Laurent series with entries from $\fs$. There is a natural non-Archimedean valuation on $k$:
$$ v(a) = \sup\{j \in \mathbb{Z}, a_i = 0 ~\forall~ i < j \}. $$
The corresponding discrete valuation ring is $O$ and $\bfk$
is its quotient field. This leads to an absolute
value $|a| = s^{-v(a)} $ which in turn induces a metric $d(a,b) = |a
- b|$ and $(k, d)$ is a separable, complete, ultrametric space. Any local field of positive
characteristic is isomorphic to a field $k$ constructed in this matter. Let $\Mat_{m \times n}(k)$ denote the $m \times n$ matrices with entries from $k$. We equip $\Mat_{m \times n}(k)$ with the $L^{\infty}$ norm which we denote $\|~\|$.  Let $\G$ be a connected, semisimple linear algebraic group, defined, isotropic and split over $\fs$ and set $G = \G(k)$. We will fix once and for all a representation $\rho : G \to \GL_{n}(k)$ and for $g \in G$ set $\|g\| :=  \|\rho(g)\|$.   Let $\Gamma$ be a lattice in $G$ and let $(\Upsilon, \mu)$ be the probability space $G/\Gamma$ equipped with Haar measure which descends from $G$.\\

\noindent We now introduce some notation from \cite{KleinMarg}. Let $\mathcal{B}$ be a family of measurable subsets of $\Upsilon$ and let $\scF = \{f_n\}$ denote a sequence of $\mu$-preserving transformations of $\Upsilon$
\begin{definition}(Borel-Cantelli families)\footnote{The definition of Borel-Cantelli families  makes sense in the setting of any probability space, and indeed this is the definition in \cite{KleinMarg}. See also \cite{Kleinbock-Chernov} for even more general formulations and applications.}
We say that $\mathcal{B}$ is \emph{Borel-Cantelli for $\scF$} if for every sequence $\{A_n~:~n \in \mathbb{N}\}$ of sets from $\mathcal{B}$, 
\begin{displaymath}
\mu(\{x \in \Upsilon~|~f_{n}(x) \in A_n~\text{for infinitely many}~n \in \N\}) 
\end{displaymath}
\begin{equation*} 
 =
\begin{cases} 0 & \text{if } \sum_{n = 1}^{\infty}\mu(A_n) < \infty,\\\\
 1 & \text{if } \sum_{n = 1}^{\infty}\mu(A_n) = \infty.
\end{cases} \end{equation*}

\end{definition}

\noindent For a function $\Delta$ on $\Upsilon$ and an integer $n \in \N$, denote by $\Phi_{\Delta}$, the tail distribution function, defined by:
\begin{equation}\label{Phi-def}
\Phi_{\Delta}(n) := \mu(\{x \in \Upsilon~|~\Delta(x) \geq s^{n}\}).
\end{equation}
   
\begin{definition}(Smooth and $\DL$)
Let $\Delta$ be a function on $\Upsilon$. We call $\Delta$ \emph{smooth} if there exists a compact open subgroup $U$ of $G$ such that $\Delta$ is $U$-invariant. For $\kappa > 0$, we say that $\Delta$ is $\kappa-\DL$ (an abbreviation for $\kappa$-\emph{ultra distance like}) if it is smooth and
\begin{equation}\label{kDL}
 \Phi_{\Delta}(n)\asymp s^{-\kappa n}~~\forall ~n \in \mathbb{Z}.
\end{equation}
Finally, we say that $\Delta$ is $\DL$ if it is smooth and there exists $\kappa > 0$ such that (\ref{kDL}) holds\footnote{In fact, we only need $ \Phi_{\Delta}(n)\ll s^{-\kappa n}~~\forall ~n \in \mathbb{Z}$. See \S \ref{section-decay}}.  
\end{definition}

\noindent We will be primarily concerned with the following two functions:
\subsection{$\Delta$ on $\SL_{r}(k)/\SL_{r}(Z)$}
For a positive integer $r$, we consider the space of unimodular lattices (co-volume $1$ $O$-submodules of maximal rank) in $k^r$, which can be identified with $\SL_{r}(k)/\SL_{r}(Z)$. On this space, we define the function:
\begin{equation}\label{delta-def}
\Delta(\Lambda) := \max_{\bv \in \Lambda \backslash \{0\}} \log_{s} \frac{1}{\|\bv\|}.
\end{equation}
\noindent The function $\Delta$ plays a crucial role in the Khintchine-Groshev theorem and its analogues.

\subsection{$d(x_0, \cdot)$ on $G/\Gamma$}
Consider the space $G/\Gamma$ where $\Gamma$ is an \emph{arithmetic} lattice in $G$. Let $X$ denote the Bruhat-Tits building of $G$ and $d$ denote the combinatorial metric on $X$. Then $d$ lifts to a right invariant pseudo-metric on $G$ and therefore also defines a pseudo-metric (continued to be called $d$) on $\Upsilon = G/\Gamma$. Fix a point $x_0$ on $\Upsilon$ and consider the function:
\begin{equation}\label{d-def}
x \to d(x_0, x).
\end{equation}
\noindent The function $d(x_0, \cdot)$ will play a crucial role in the derivation of logarithm laws.

\noindent It turns out  (see \S \ref{section-volumes}) that

\begin{theorem}\label{theorem-volumes}
The functions $\Delta$ and $d(x_0, \cdot)$ are $\DL$.
\end{theorem}  

\noindent We refer the reader to \cite{KleinMarg} for the concept of ``distance like" functions. The above definition is a natural ultrametric replacement for distance like.
\subsection{Main Results}

Our first result is

\begin{theorem}\label{km1.8}
Let  $\scF = \{f_n~|~n \in \N\}$ be a sequence of elements of $G$ satisfying
\begin{equation}\label{ED}
\sup_{m \in \N} \sum_{n = 1}^{\infty} \|f_n \fmi\|^{-\beta} < \infty~\forall~\beta > 0,
\end{equation}

\noindent and let $\Delta$ be a $\DL$ function on $\Upsilon$. Then
\begin{equation}\mathcal{B}(\Delta) := \left\{ \{ x \in \Upsilon~|~\Delta(x) \geq s^n\}~|~n \in \mathbb{Z} \right\}\end{equation}
\noindent is Borel-Cantelli for $\scF$.

\end{theorem}

\noindent With appropriate choices of $\Delta, G$ and $\Gamma$, Theorem \ref{km1.8} implies positive characteristic versions of logarithm laws as well as several results related to Khintchine's theorem.  Sequences which satisfy (\ref{ED}) above are referred to as ``exponentially divergent" (abbreviated $\ED$) in \cite{KleinMarg}. We now turn our attention to asymptotics.

\begin{theorem}\label{km4.3}
Let $G, \Gamma, \scF$ and $\Delta$ be as before. Let $\{r_n\}$ be a sequence of integers such that
\begin{equation}\label{km4.31}
\sum_{n = 1}^{\infty} \Phi_{\Delta}(r_n) = \infty.
\end{equation}
\noindent There exists a positive constant $c \leq 1$ such that for almost every $x \in G/\Gamma$
\begin{align}\label{km4.32}
c &\leq \liminf_{N \to \infty} \frac{\#\{1 \leq n \leq N~:~ \Delta(f_n x) \geq s^{r_n}\}}{\sum_{n = 1}^{N}\Phi_{\Delta}(r_n)}\\
  & \leq \limsup_{N \to \infty} \frac{\#\{1 \leq n \leq N~:~ \Delta(f_n x) \geq s^{r_n}\}}{\sum_{n = 1}^{N}\Phi_{\Delta}(r_n)} \leq  \frac{1}{c}.
\end{align}
\noindent Moreover, if $G$ is assumed to be center free then $c$ can be chosen to be $1$, i.e. for almost every $x \in G/\Gamma$,
\begin{equation}\label{km4.33}
\lim_{N \to \infty} \frac{\#\{1 \leq n \leq N~:~ \Delta(f_n x) \geq s^{r_n}\}}{\sum_{n = 1}^{N}\Phi_{\Delta}(r_n)} = 1.
\end{equation}
\end{theorem}

\noindent Systems which satisfy (\ref{km4.33}) are referred to as \emph{Strongly Borel Cantelli} \cite{Kleinbock-Chernov}. Note that for semisimple real Lie groups, this stronger than Borel Cantelli property is not known. The reason our result is more precise is a manifestation of the ultrametric.

\subsection{Remarks and Plan of the Paper}
This paper should be considered a descendent of \cite{KleinMarg} - essentially we are following broadly their method in the setting of local fields of positive characteristic. Given this we have cited instead of reproducing, most results which can be obtained mutatis mutandis from \cite{KleinMarg} or other works. Instead we have  concentrated on those results which need new ideas, or where the transition from Lie groups to algebraic groups over local fields of positive characteristic is not obvious.\\

The results of this paper were announced in \cite{AGP}.  Our original intention as stated in the announcement, was to  to provide a unified treatment of the logarithm laws for local fields in arbitrary characteristics. However, this project proved to be far too diverse to accumulate in a single work. The ultrametric nature of the positive characteristic case makes the results stronger and our proofs quite different. We have therefore decided to split up our work accordingly and will pursue the characteristic zero case in a separate work.\\


In the announcement, we did not assume that the group in question was split. The main theorems in this paper, namely Theorems \ref{km1.8} and \ref{km4.3} are indeed valid without this assumption. Theorem \ref{main-mixing} also holds but we are currently unable to prove Theorem \ref{volume-at-cusps} for non-split groups.\\

The results of this paper remain valid when we consider more than one valuation simultaneously, i.e. the so-called $S$-arithmetic setting. A remark about our title, the reason we have not treated the case $k = \mathbb{Q}_{p}$, which is also ultrametric is that, by a result of Tamagawa, lattices in semisimple $p$-adic algebraic groups are always co-compact. Since our results are manifestations of cuspidal excursions of geodesic trajectories, they are not interesting in the co-compact case.\\

In \S \ref{section-decay}, we state and prove the effective decay of matrix coefficients for semisimple groups over local fields of positive characteristic, in \S \ref{section-main} we prove our main theorems, \S \ref{section-volumes} is devoted to calculating volumes of complements of compact sets in homogeneous spaces, \S \ref{section-diophantine} to Khintchine's theorem and generalizations, \S \ref{section-geodesic} to logarithm laws for geodesics and \S \ref{section-concluding} to final remarks and some extensions.

\subsection{Acknowledgements} We thank S. Mozes and F. Paulin for helpful conversations. We thank the LMS for a scheme 2 grant which allowed J. S. A. to visit Norwich.

\section{Decay of matrix coefficients}\label{section-decay}

In this section, we will prove the effective decay of matrix coefficients of semisimple groups. In addition to the hypotheses from the introduction, we will assume that $G$ has trivial center. We denote by $L^{2}_{0}(G/\Gamma)$ the subspace of $L^{2}(G/\Gamma)$ orthogonal to constant functions. The regular representation of $G$ on $L^{2}_{0}(G/\Gamma)$ will be denoted by $\rho_0$. The main result in this section is:

\begin{theorem}\label{main-mixing}
Let $\bbG$ be a connected, semisimple, linear algebraic group, defined and split over $\bfk$ and $\Gamma$ be an non-uniform irreducible lattice in $G$. There exist constants $a > 0$ and $b \in \mathbb{Z}_{+}$ such that for any $g \in G$ and any smooth functions $\phi, \psi \in L^{2}_{0}(G/\Gamma)$,
\begin{equation}\label{HCbound}
|\langle \rho_{0}(g)\phi, \psi \rangle | \leq a \|\phi\|_{2}\|\psi\|_{2} \|g\|^{-1/b}.
\end{equation} 
\noindent The constants depend on $G$, $\Gamma$ and the choice of compact subgroup $U$ in the definition of the smooth vectors.

\end{theorem}

\noindent We note that Theorem \ref{main-mixing} is almost certainly well known to experts and indeed there exist reasonably close statements, for example (\cite{KleinMarg}, \cite{GMO}, \cite{EinMargVen}, \cite{Oh}). In particular, it should be noted that the last work establishes the best possible estimates for decay of matrix coefficients in many cases. However, we could not explicitly identify a reference and therefore provide a proof assembled from the literature. 

Decay rates for matrix coefficients play a vital role in the study of group actions on homogeneous spaces. For example the following corollary (known as the Howe-Moore theorem) implies that the $G$ action on $G/\Gamma$ is mixing.

\begin{corollary}
The matrix coefficients of $\rho_0$ ``vanish at infinity", i.e. for any $\phi, \psi \in L^{2}_{0}(G/\Gamma)$ and any sequence $g_n \to \infty$ in $G$,
$$|\langle \rho_{0}(g_n)\phi, \psi \rangle | \to 0.$$
\end{corollary}

\noindent We will denote by $\widehat{G}$ the unitary dual of $G$ - namely the set of equivalence classes of irreducible unitary representations of $G$. The trivial representation will be denoted by $1_{G}$.
\begin{definition}(Almost invariance, Spectral gap and Strong spectral gap )
A unitary representation $(\sigma, \mathcal{H})$ of $G$ has \emph{almost invariant vectors} if, for every compact
subset $Q$ of $G$ and every $\epsilon > 0$, there exists a unit vector $\xi \in \mathcal{H}$ such that
 $$\sup_{g \in G} \|\sigma(g)\xi - \xi\| < \epsilon.$$
\noindent $(\sigma, \mathcal{H})$ is said to have \emph{spectral gap} if it has no almost invariant vectors, and is said to have \emph{strong spectral gap} if the restriction of $\sigma$ to every almost simple factor $G_i$ of $G$  has spectral gap.
\end{definition}

\noindent An equivalent definition of spectral gap, also popular in the literature can be made as follows. Let $\sigma, \rho$ be unitary representations of $G$. Then $\rho$ is \emph{weakly contained} in $\sigma$ if every matrix coefficient of $\sigma$ is a limit, uniformly on compact subsets of $G$ of positive linear combinations of matrix coefficients of $\rho$. We will denote this by $\rho \prec \sigma$.  We say that $1_G$ is isolated from a subset $R$ of $\widehat{G}$ if it is not weakly contained in any representation of $R$. From \cite{BdVa} Corollary F.1.5, it follows that $\sigma$ has almost invariant vectors if and only if $1_{G} \prec \sigma$, i.e.  $$\sigma~\text{has spectral gap}~\iff 1_{G} \notin \Supp \sigma. $$
\noindent Here the support $\Supp \sigma$ of $\sigma$ is the set of all representations in $\widehat{G}$ which are weakly contained in $\sigma$.

Let $\bbA$ be a maximal $k$-split torus of $\bbG$ and $\bbB$ be a minimal parabolic subgroup of $\bbG$ containing $\bbA$. Let $\Phi$ be the set of simple roots of $\bbA$ in $\bbG$ with ordering given by $\bbB$, let $\Phi^{+}$ denote the set of positive roots and let
$$ A^{+} := \{a \in A~:~|\chi(a)| \geq 1~\text{for all}~\chi \in \Phi^{+}\}.$$

\noindent Let $\Delta_B$ denote the modular function of $B$. For $a \in A^{+}$, we have the formula
\begin{equation}\label{formula-Delta}
 \Delta_{B}(a) = \prod_{\text{positive}~\chi \in \Phi} |\chi(a)|^{m_{\chi}}
\end{equation} 

\noindent where $m_{\chi}$ is the multiplicity of $\chi$. The Harish Chandra function $\Xi$ of $G$ is the spherical function defined as 
\begin{equation}\label{Harish}
\Xi_{G}(g) = \int _{K} \Delta_{B}(p(gk))^{-1/2}~d\eta(k)
\end{equation}
\noindent where $\eta$ is Haar measure on $K$ and $p : G \to P$ is projection. Alternatively, $\Xi_{G}$ can be defined as the diagonal matrix coefficient of a unit vector for the regular representation of $G$ on $L^{2}(G/B)$. It is known that $\Xi_{G}$ is a continuous $K$ bi-invariant function on $G$. 



\noindent  There is a convenient bound for the Harish-Chandra function in terms of $\|g\|$ (cf. \cite{GMO} Lemma 3.6).

\begin{proposition}
There exists constants $\sigma \in \mathbb{Z}_{+}$ and $\varsigma > 0$ such that for any $g \in G$
\begin{equation}\label{HCboundfinal}
\Xi_{G}(g) \leq \varsigma \|g\|^{-1/\sigma}. 
\end{equation}
\end{proposition}

\noindent The proof of Theorem \ref{main-mixing} is based on the following results of Cowling, Haagerup and Howe (\cite{CHH}, see also \cite{GN1} \S 5 for a nice exposition of this circle of ideas)
\begin{definition}(Strongly $L^{r}$ representations)
A unitary representation $(\sigma, \mathfrak{H})$ of $G$ is called strongly $L^{r}$ if there exists a dense subset of vectors $\phi, \psi \in \mathfrak{H}$ such that the matrix coefficient
$ \langle \sigma(g)\phi,\psi\rangle$
is in $L^{r+\epsilon}(G)$ for every $\epsilon > 0$. 
\end{definition}

\noindent We now state in a form convenient for us, Theorems from $5.4$ and $5.6$ from \cite{GN1} which in turn summarize results from \cite{CHH} and \cite{Cowling}, \cite{Howe-Moore} and \cite{Borel-Wallach}. 
\begin{theorem}
Let $(\sigma, \mathfrak{H})$ be a unitary representation of $G$. Let $v$ and $w$ be 
two $K$-finite vectors in $\mathfrak{H}$ and denote the dimensions of their spans under $K$ by $d_v$ and 
$d_w$. Then 
\begin{enumerate}
\item If $\sigma$ is weakly contained in the regular representation, then 
\begin{equation}
|\langle \sigma(g)v, w \rangle | \leq \sqrt{d_v d_w} \|v\|\|w\|\Xi(g)
\end{equation}
\item If $\sigma$ is strongly $L^{2r}$, then 
\begin{equation}
|\langle \sigma(g)v, w \rangle | \leq \sqrt{d_v d_w}\|v\|\|w\|\Xi(g)^{1/r}.
\end{equation}
\end{enumerate}
\end{theorem}

\begin{theorem}(Cowling-Haagerup-Howe)
Let $(\sigma, \mathcal{H})$ be a unitary representation of $G$ with strong spectral gap. Then  $(\sigma, \mathcal{H})$ is strongly $L^{r}$ for some $r < \infty$. 
\end{theorem}

\noindent Note that if the $k$-rank of each almost simple factor of $G$ is at least $2$ then $G$ has Kazhdan's property $(T)$, which means that $1_G$ is isolated in $\widehat{G}$ and this accounts for the fact that the property of being strongly $L^r$ is a property of the group and not of the representation.\\

\noindent Our task is therefore reduced to showing
\begin{theorem}\label{ssg}
Let $\bbG$ be as in Theorem \ref{main-mixing} and $\Gamma$ be an non-uniform irreducible lattice in $G$. Then  $\rho_0$, the regular representation of $G$ on $L^{2}_{0}(G/\Gamma)$ has strong spectral gap.
\end{theorem}

\noindent If $\bbG$ is a simple group, then a result of Bekka-Lubotzky \cite{BekkaLubotzky} tells us that $\rho_0$ has spectral gap. We may thus assume that it has at least two isotropic, almost simple factors. By Margulis' arithmeticity theorem, $\Gamma$ is an arithmetic (\cite{Margulisbook} IX 1.4 and \cite{Venkataramana}) lattice in $G$. This means that there exists a global field $\bfk'$, a connected, semisimple, adjoint  $\bfk'$-group $\bbH$  and a continuous homomorphism $\phi : H \to G$ of class $\Psi$ (see \cite{Margulisbook} IX for the definition) such that $\phi(\bbH(Z'))$ is commensurable with $\Gamma$. 

The spaces $G/\Gamma$ and $G/\phi(\bbH(Z'))$ thus have a common finite covering and we are in a position to use \cite{KleinMarg}, Lemma 3.1 to assume without any loss of generality that $\bbG$ is adjoint and $\Gamma = \bbG(Z)$. By (\cite{Margulisbook} I.1.7), for every adjoint almost $\bfk$ simple group $\bbG_i$, there exists a finite separable extension $\bfk'_{i}$ of $\bfk$ and a connected adjoint absolutely almost simple $\bfk'_{i}$ group $\bbG'_{i}$ such that $\bbG_i = R_{\bfk'_{i}/\bfk}(\bbG'_{i})$ where $R_{\bfk'_{i}/\bfk}$ denotes Weil's restriction of scalars functor. We may thus assume that the factors of $\bbG$ are absolutely almost simple.

\noindent It is well known (\cite{Margulisbook}, Proposition I.1.6.3) that there exists a $\bfk$ morphism $\mathbb{SL}_2 \to \bbG_i$ with finite kernel. We denote the image of $\mathbb{SL}_2$ by $\bbM_i$. By, \cite{Clozel1} Lemma 3.6, it is enough to show that the restriction of $\rho_0$ to any $M_i$ has spectral gap. Following \cite{BurSar} the spectrum of $G/\Gamma$ is defined to be the set of all representations in $\widehat{G}$ which are weakly contained in $L^{2}(G/\Gamma)$. The automorphic spectrum $\widehat{G}^{\aut}$ is defined to be the closure in the Fell topology of the union of the spectra of $G/\Gamma$ as $\Gamma$ varies over congruence subgroups of $G$. By the Burger-Sarnak restriction principle (see \cite{BurSar} for real groups and \cite{ClozelUllmo} for $p$-adic groups. As has been observed in \cite{Clozel2}, the proof in \cite{ClozelUllmo} works verbatim for local fields of positive characteristic) it is enough to prove that the trivial representation is isolated in $\widehat{M}^{\aut}$. It follows from the important work of Drinfeld \cite{Drinfeld}, settling the Ramanujan-Petersson conjecture for $\mathbb{GL}_2$ that this is indeed the case, thereby completing the proof of Theorem \ref{main-mixing}.\\

\noindent As we have remarked, Theorem \ref{main-mixing} is probably well-known to experts, although we were unable to identify a source. However, effective mixing estimates have found extensive applications in ergodic theory and we hope that this result will be of independent interest. Also we note that in the rank $1$ situation, an effective rate for the decay of matrix coefficients for the geodesic flow is known in the much more general setting of CAT(-1) spaces due to work of Liverani \cite{Liverani}.

\section{Proof of Main Theorems}\label{section-main}
\subsection{Borel-Cantelli lemma}\label{SBC}
For a sequence of functions $\scH = \{h_n\}$ on $(G/\Gamma, \mu)$, we define
\begin{equation}
S_{\scH, N}(x) := \sum_{n = 1}^{N}h_{n}(x)~\text{and}~E_{\scH, N} := \int_{G/\Gamma} S_{\scH, N}~\ud\mu
\end{equation}

\noindent We recall the following result from \cite{Sprindzhuk} due to V. G. Sprindzhuk attributed to W. M. Schmidt, which provides a converse to the Borel-Cantelli lemma under a suitably weak (i.e. checkable in practice) independence condition.

\begin{proposition}\label{BClemma}
Let $\scH = \{h_n ~:~n \in \N\}$ be a sequence of functions on $(G/\Gamma, \mu)$  satisfying the following two conditions:

\begin{equation}\label{BC1}
\int_{G/\Gamma}h_n~\ud\mu \leq 1~\text{for every}~n \in \N~\text{and}
\end{equation}
\noindent there exists $C > 0$ such that 
\begin{equation}\label{BC2}
\sum_{m,n = M}^{N}  \left(\int_{G/\Gamma}h_m h_n~\ud\mu - \int_{G/\Gamma}h_m \ud\mu \int_{G/\Gamma} h_n \ud\mu\right) \leq C \sum_{n = M}^{N}\int_{G/\Gamma}h_n \ud\mu \ud n
\end{equation}
\noindent for all $N > M \geq 1$. Then for $\mu$ almost every $x \in G/\Gamma$, 

\begin{equation}\label{BC3}
\lim_{N \to \infty}\frac{S_{\scH, N}(x)}{E_{\scH, N}} = 1
\end{equation}
\noindent whenever $E_{\scH, \infty}$ diverges.
\end{proposition}

\noindent We will also need the following Lemma (cf. \cite{KS}, \cite{KleinMarg} Lemma $2.3$)
\begin{lemma}\label{BClemma2}
Let $\scH$ be a sequence of functions on $(G/\Gamma, \mu)$.  Then
$$\liminf_{N \to \infty}\frac{S_{\scH, N}(x)}{E_{\scH, N}} < \infty~\text{for}~\mu~\text{almost every}~x \in G/\Gamma.$$
\noindent In particular, $\scH$ is summable whenever $S_{\scH, \infty}$ is finite almost everywhere.  
\end{lemma}

\noindent We now provide a 
\begin{proof}[Proof of Theorem \ref{km4.3}.] Accordingly, let $\Delta$ be a $\DL$ function on $G/\Gamma$ and $\scF = \{f_n\}$ be an $\ED$ sequence in $G$. Recall that we have to show that 
$$\mathcal{B}(\Delta) := \left\{ \{ x \in \Upsilon~|~\Delta(x) \geq s^n\}~|~n \in \mathbb{Z} \right\}$$
\noindent is Borel Cantelli for $F$. Let $h_n$ be a $\DL$ function on $L^{2}(G/\Gamma)$. The plan is to apply Lemma \ref{BClemma} to the \emph{twisted} sequence
$$ \scHF := \{\fni h_n~:~n \in \N\}.$$
\noindent We focus on the left hand side of the inequality (\ref{BC2}), namely 
\begin{equation}\label{BC2new}
\sum_{m,n = M}^{N}\left( \langle \fmi h_m, \fni h_n\rangle - \int_{G/\Gamma}h_m \ud\mu \int_{G/\Gamma} h_n \ud\mu\right) 
\end{equation}

\noindent  Below, we subsume various constants into Vinogradov notation. Note that in view of (\ref{HCbound}) and the $\mu$-invariance of $\scF$ we have that (\ref{BC2new}) is bounded above by
\begin{align}
 &\ll \sum_{m,n = M}^{N}\|h_m\|_2 \|h_n\|_2  \|f_m \fni \|^{-1/\sigma} \nonumber\\
 &\ll \sum_{m,n = M}^{N} s^{\frac{-\kappa n}{2}} s^{\frac{-\kappa m}{2}}  \|f_m \fni\|^{-1/\sigma} \nonumber\\
 &\ll \sum_{m,n = M}^{N} s^{\frac{-\kappa n}{2}} s^{\frac{-\kappa m}{2}}   \sum_{m,n = M}^{N}\|f_m \fni \|^{-1/\sigma} \nonumber\\
 &\ll \left( \sum_{m = 1}^{N} s^{-\kappa m}  \right)\sum_{m,n = M}^{N}\|f_m \fni\|^{-1/\sigma}\nonumber\\
 &\ll \sup_{n \in \N} \sum_{m = 1}^{\infty}\|f_m \fni\|^{-1/\sigma} E_N.
\end{align}
\noindent Since the right hand side is finite, we have demonstrated (\ref{km4.33}).\\

\noindent We note that Theorem \ref{km1.8} is an immediate consequence of Theorem \ref{km4.3} and Lemma \ref{BClemma2}. One can now use the same strategy as \cite{KleinMarg} \S 4.3 to pass from center-free groups to groups with non-trivial center. Let $Z$ denote the center of $G$, $G'$ denote $G/Z$ and $p : G \to G'$ denote the natural homomorphism. In loc. cit. it is shown that $\DL$ functions $\Delta'$ on $G'/p(\Gamma)$ arise from $\DL$ functions $\Delta$ on $G/\Gamma$ and moreover that $\Delta/\Delta'$ is uniformly bounded between two positive constants. Thus (\ref{km4.32}) reduces to the center-free case, losing at the most a multiplicative constant.

\end{proof}

\section{Volume estimates for cusps}\label{section-volumes}
In this section, we will prove Theorem \ref{theorem-volumes}.
\subsection{$d(x_0, \cdot)$ is $\DL$}
Recall that in order to prove that $d(x_0, \cdot)$ is $\DL$, we need to show that
\begin{equation}\nonumber
\mu(\{x \in \Upsilon~:~ d(x_0, x) \geq s^n\}) \asymp s^{-\kappa n},
\end{equation}
\noindent where $x_0$ is a fixed point on $G/\Gamma$, $\Gamma$ is an arithmetic lattice in $G$, $\kappa$ is the $k$ rank of $\bbG$ and $d(x_0, x)$ is the distance on $\Upsilon := G/\Gamma$ of $x$ from $x_0$. Using the arithmeticity theorem of Margulis-Venkataramana, we may assume that $\Gamma = \bbG(Z)$.\\
 We fix a Borel subgroup $\bbB$, defined over $\fs$, that contains a maximal split torus $\bbT$ of $\bbG$ and denote by $X_*(T)$ the lattice $\Hom(\Gm,\bbT)$ of cocharacters of $\bbT$. Fix a section $W \to N_{\bbG(\fs)} (\bbT)$ and denote the image of $w\in W$ under this section also by $w$. Define the map $\phi: W \times X_*(\bbT) \to G$ by
\begin{equation}\label{eq:embedding}
  \phi:(w,\mu)\mapsto w X^{-\mu}.
\end{equation}
\noindent Here $\bbG(\fs)$ is viewed as a subgroup of $G$. For each $T> 0$ we denote by $\Upsilon_T$ the subset of $G/\bbG(Z)$ consisting of the union of $\bbG(O)\phi(e^\lambda)\bbG(Z)$, where $\lambda$ ranges over the dominant cocharacters of $\bbT$ for which $\langle \rho,\lambda \rangle\geq T$. Then to show that $d(x_0,~)$ is $\DL$, for the case $\Gamma = \bbG(Z)$ it is enough to show that
\begin{theorem}
  \label{volume-at-cusps}
  Let $r$ be the rank of $G$. Then for every $T>0$,
  \begin{equation*}
    \mu(\Upsilon_T) \asymp \sum_{l\geq T} s^{-l}l^{r-1}.
  \end{equation*}
\end{theorem}

\noindent This will be the main result in this section.

We define $I$ to be the pre-image of $\bbB(\fs)$ under the canonical map $\bbG(O)\to \bbG(\fs)$, $I$ is an \emph{Iwahori subgroup} of $G$. Let $\bar \bbB$ denote the Borel subgroup in $\bbG$ opposite to $\bbB$ (thus $\bar \bbB \cap \bbB = T$).
Let $\Gamma$ denote the pre-image in $\bbG(Z)$ of $\bar \bbB(\fs)$ under the canonical map $\bbG(Z) \to \bbG(\fs)$.
$\Gamma$ is of finite index in $\bbG(Z)$ and is therefore a non-uniform arithmetic lattice.
Recall from \cite[Chapter 2]{thesis} (part of this also appears in \cite{MR1983748}) the following results:
\begin{theorem}
  \label{theorem:Birkhoff}
  The map $\phi$ induces bijections 
  $$W\times X_*(\bbT)\longleftrightarrow I \bsl G/\Gamma$$ 
  \begin{center}
  and
  \end{center}
  $$X_*(\bbT)^{++}\longleftrightarrow \bbG(O)\bsl G/\bbG(Z).$$
\end{theorem}
In order to better understand the double coset $I\phi(w,\mu)\Gamma$, it is necessary to endow $W\times X_*(\bbT)$ with some additional structure.
Note that $W$ acts on $X_*(\bbT)$.
One may, therefore, form the semidirect product
\begin{equation*}
  \widetilde W = W\ltimes X_*(\bbT).
\end{equation*}
$\tilde W$ is known as the \emph{extended affine Weyl group} of $\bbG$ with respect to $\bbT$.
The function $\phi_v$ induces an isomorphism of $\tilde W$ with $T(O_v)\bsl N_G(T(F_v))$. 
Let $Q$ denote the sublattice of $X_*(T)$ generated by the coroots of $T$ in $G$.
The subgroup $W_a=W\ltimes Q$ of $\tilde W$ is known as the \emph{affine Weyl group} of $G$ with respect to $T$.
Its structure, and its exact relation to $\tilde W$, can be described as follows: let $\Delta=\{\alpha_1,\ldots, \alpha_r\}$ be the base for the root system of $T$ in $G$ determined by $B$ and let $s_i$ denote the reflection in $X_*(T)\otimes \R$ about the hyperplane on which $\alpha_i$ vanishes.
If we assume that $G$ has a simple root system, then there is a unique highest root $\tilde\alpha$ with respect to the base $\Delta$.
Let $s_0$ denote the affine reflection about the affine hyperplane $\tilde \alpha+1=0$ and let $\tilde S =\{s_0,\ldots,s_r\}$.
Then $(W_a,\tilde S)$ is a Coxeter system.
$\Pi=X_*(T)/Q$ is a finitely generated abelian group, and there is a short exact sequence
\begin{equation*}
  1\to W_a \to \tilde W \to \Pi \to 1.
\end{equation*}
The action of $\Pi$ on $W_a$ is length preserving.
Therefore, the length function $l$ of the Coxeter group $W_a$ extends to a function $\tilde W\to \Z_{\geq 0}$ which will also be denoted by $l$.

Denote by $w_0$, the longest element of $W$ (relative to the base $\Delta$). 
\begin{theorem} \cite[Section~5.b]{MR2003e:11057}
  \label{theorem:volume-of-Iwahori-double-coset}
  For every $\tilde w \in \tilde W$,
  \begin{equation*}
    \mu(l\phi(\tilde w)\Gamma)=q^{-l(\tilde w)}
  \end{equation*}
\end{theorem}
\begin{proposition}
  \cite[Section~6.1]{MR2250034}
  For each $\lambda\in X_*(T)$, denote by $e^\lambda$ its image in $\tilde W$.
  If $\lambda\in X_*(T)^{++}$, then $l(e^\lambda)=\langle \rho, \lambda\rangle$, where $\rho$ is the sum of all roots that are positive with respect to $B$.
\end{proposition}
\noindent Under the natural surjection
\begin{equation*}
  I \bsl \bbG(k)/\Gamma \to \bbG(O) \bsl \bbG(k)/\bbG(Z),
\end{equation*}
the fiber over the double coset of $\phi(e^\lambda)$, when $\lambda\in X_*(T)^{++}$, consists precisely of the double cosets of $\phi(we^\mu)$, where $w\in W$, and $\mu$ lies in the $W$-orbit of $\lambda$.
If $\lambda=w_2\cdot \mu$, then $we^\mu=w_1e^\lambda w_2$, where $w_1=ww_2^{-1}$.
Hence we conclude that the fiber contains at most $|W|^2$ elements, and that their lengths are within $2l(w_0)$ of $l(e^\lambda)$.
It follows that there exist constants $C_1$ and $C_2$ independent of $\lambda$ such that 
\begin{equation}
  \label{eq:coset-volume-bound}
  C_1q^{-\langle \rho, \lambda\rangle}<\mu(\bbG(O)\phi(e^\lambda)\bbG(Z))<C_2q^{-\langle \rho, \lambda\rangle}
\end{equation}

\begin{proof}
  The theorem follows from (\ref{eq:coset-volume-bound}), once it is observed that the number of dominant cocharacters $\mu$ such that $\langle \rho,\lambda \rangle=l$ grows as $l^{r-1}$.
\end{proof}

\subsection{$\Delta$ is $\DL$.}

The main tool in the proof that $\Delta$ is $\DL$ in \cite{KleinMarg} is reduction theory on $\SL_{r}(\R)/\SL_{r}(\Z)$, namely the following famous result of Siegel \cite{Siegel} and a multidimensional generalization (Theorem $7.3$ \cite{KleinMarg}). Given a lattice $\Lambda$ in $\R^r$ and a function $\phi \in L^{1}(\R^r)$, we define
$$ \hat{\phi} := \sum_{\bv \in \Lambda} \phi(\bv)$$

\begin{theorem}
Let $\phi \in L^{1}(\R^r)$, then $\hat{\phi} \in L^{1}(\SL_{r}(\R)/\SL_{r}(\Z))$ and
$$ \int_{\R^r} \phi ~d\bv = \int_{\SL_{r}(\R)/\SL_{r}(\Z)} \hat{\phi} ~d\mu.$$
\end{theorem}

Siegel's result has been generalized widely (\cite{Weil}, \cite{Morishita}).  Let $(X, \mu)$ be a topological space equipped with a probability measure on which a unimodular group $G$ acts transitively preserving $\mu$. Let $L$ be a discrete subspace of $X$ on which a discrete subgroup $\Gamma$ of $G$ acts stably. Let $H$ be the stabilizer of $x \in L$ and $\Gamma_0 = H \cap \Gamma$. Then in \cite{Morishita} it is shown that there exists a positive constant $c$ such that for any $f \in L^{1}(X)$
\begin{equation}\label{Weil}
\int_{X} f(x)~d\mu = \int_{G/\Gamma} \left( \sum_{z \in \Gamma/\Gamma_0}f(gz) \right)~dg.
\end{equation}

\noindent Clearly the above statement includes our special case when $X = k^r$, $G = \SL_{r}(k)$ and $\Gamma = \SL_{r}(Z)$.  Following the proof of Theorem $7.3$ in \cite{KleinMarg}, it is straightforward to obtain a multidimensional version of the result. This positive characteristic version of Siegel's result along with its multidimensional version  can then be used to show that $\Delta(x_0, \cdot)$ is $\DL$. Since the proof is a verbatim repetition, we omit it.

\section{Diophantine results}\label{section-diophantine}

\subsection{The Khintchine-Groshev Theorem}
The archetypal result in the theory of metric Diophantine approximation is the Khintchine-Groshev theorem (\cite{Khintchine} and \cite{Groshev}). For the rest of this paper, we fix a non-increasing continuous function $\psi$ from $\R_{+} \to \R_{+}$. Let $\Mat_{\mn}(k)$ denote the set of $\mn$-matrices with real valued entries, $\|~\|$ denote the $L^{\infty}$ norm, $\Imn$ be the ball $\{A \in \Mat_{\mn}(k)~:~\|A\| \leq 1\}$. We  denote by $\lambda$, Lebesgue measure on $\Mat_{\mn}(k)$ normalized so that $\lambda(\Imn) = 1$. Let $\mathcal{W}_{\mn}(\psi, k)$ be the subset of $\Imn$ for which there exist infinitely many $\bq \in Z^n$ such that 
\begin{equation}\label{psi1} 
\|\bp + A\bq\|^m < \psi(\|\bq\|^n)
\end{equation}

\noindent for some $\bp \in Z^m$. The Khintchine-Groshev theorem characterizes the size of $\mathcal{W}_{\mn}(\psi, k)$ in terms of Lebesgue measure. This size depends on the sum
\begin{equation}\label{int}
\sum_{x = 1}^{\infty} \psi(x)
\end{equation}

\noindent Our first theorem is

\begin{theorem}\label{KG}

\begin{equation*} 
\lambda(\mathcal{W}_{\mn}(\psi, k)) =
\begin{cases} 0 & \text{if } (\ref{int})~ converges ,\\\\
 1 & \text{if } (\ref{int})~ diverges.
\end{cases} \end{equation*}

\end{theorem}

\noindent The substance of the Khintchine Groshev theorem lies in the divergence part, the convergence statement being an immediate consequence of the Borel Cantelli Lemma. We provide now a brief history. Replace $k$ with $\R$ and $Z$ with $\Z$. In this, the most studied setting, the above theorem was established by Khintchine \cite{Khintchine} for the case $m = n = 1$ and by Groshev \cite{Groshev} in higher dimensions. Diophantine approximation in function fields has been the subject of much recent activity (see for instance the surveys \cite{Lasjaunias}, \cite{Schmidt} as well as the papers \cite{Caulk-Schmidt}, \cite{Ghosh}, \cite{Kristensen}, \cite{deMathan}). Reverting back to our setup, Theorem \ref{KG} is also known, in fact in slightly greater generality. In \cite{Kristensen}, this is proved without the continuity assumption on $\psi$ and in \cite{Inoue}, the monotonicity condition on $\psi$ is relaxed. Our method of proof is different.

\noindent However, our second result, a multiplicative generalization of Khintchine's theorem is new.  In order to set the stage for this as well as the proof of Theorem \ref{KG}, we recall some terminology from \cite{KleinMarg}. For integers $m, n$, set $\Upsilon_{m+n} := \SL_{m+n}(k)/\SL_{m+n}(Z)$, the space $\Upsilon_{m+n}$ can be identified with the non-compact space of unimodular (i.e. co-volume $1$) lattices (i.e. free $Z$ modules) in $k^{m+n}$. For a lattice $\bv$ in $k^{m + n}$ we denote by $\bv^{(m)}$ the vector comprising its first $m$ coordinates and by $\bv_{(n)}$ the vector comprising the last $n$. We call a lattice $\Lambda \in \Upsilon_{m+n}$, $(\psi, n)$-approximable if there exist infinitely many $\bv$ with arbitrarily large $\| \bv_{(n)}\|$ such that
\begin{equation}\label{latt-approx}
\| \bv^{(m)}\|^{m} \leq \psi(\|\bv_{(n)}\|^n)
\end{equation}

\noindent The connection between the two definitions is as follows. Given a matrix $A \in \Mat_{\mn}(k)$, we associate to it a lattice $\Lambda_{A}$ in $k^{m+n}$ defined by
$$ \Lambda_A :=  \begin{pmatrix} I_{m} & A \\ 0 & I_n \end{pmatrix} Z^{m + n} = \begin{pmatrix}A\bq + \bp\\\bq\end{pmatrix}.$$
\noindent Then $A$ is $\psi$ approximable if and only if $\Lambda_{A}$ is $(\psi, n)$ approximable. We will prove

\begin{theorem}\label{KG-gen}
Almost every lattice $\Lambda$ in $\Upsilon_{m + n}$ is $(\psi, n)$ approximable provided
\begin{equation}
\sum_{x = 1}^{\infty} \psi(x)
\end{equation}
\noindent diverges.
\end{theorem}
\noindent Theorem \ref{KG} follows from this using a Fubini type argument (cf. \cite{KleinMarg}, 8.7 and \cite{Dani}, 2.1) which we will not repeat. For $n \geq 2$ and $\bv = (v_1, \dots, v_n) \in k^n$, we set
$$ \prod(\bv) := \prod_{i = 1}^{n} |v_n|~\text{and}~\|\bv\| := \max_{i}|v_i|.$$
\noindent Call $\Lambda \in \Upsilon_r$ , $\psi$ \emph{multiplicatively approximable} if there exist infinitely many $\bv \in \Lambda$ such that 
\begin{equation}\label{mult-approx}
\prod(\bv) \leq \|\bv\| \psi(\|\bv\|).
\end{equation}

\noindent For $r > 0$, set $\psi_{r}(x) := 1/x(\log x)^{r}$. In \cite{Skriganov}, M. Skriganov investigated the size of the set of $\psi_r$-multiplicatively approximable lattices. In particular, he was interested in the difficult problem of determining the existence and precise value of the exponent $r$ which makes the property of being $\psi_r$ multiplicatively approximable, generic in the space of lattices. Problems of this type are of great interest in the geometry of numbers. Moreover, multiplicative results have historically proved more difficult than their non-multiplicative counterparts. In \cite{KleinMarg}, the following theorem was proved, answering Skriganov's question.

\begin{theorem}(\cite{KleinMarg} Theorem $1.11$)
Almost every (resp. almost no) unimodular lattice $\Lambda$ in $\bbR^n$ ($n \geq 2$) is $\psi$ multiplicatively approximable provided 
$$ \sum_{x = 1}^{\infty} (\log x)^{n - 2} \psi(x)$$
\noindent diverges (resp. converges).
\end{theorem}

\noindent We will prove 

\begin{theorem}\label{Skriganov}
Almost every lattice $\Lambda$ in $k^n$ is $\psi$ multiplicatively approximable provided 
$$ \sum_{x = 1}^{\infty} (\log x)^{n - 2} \psi(x)$$
\noindent diverges.
\end{theorem}

\subsection{Proof of Theorems \ref{KG} and \ref{Skriganov}}

The proofs of Theorems \ref{KG-gen} and \ref{Skriganov} depend on a dynamical reformulation of Diophantine properties going back to Dani \cite{Dani} and generalized by Kleinbock-Margulis. We recall first Lemma $8.3$ from \cite{KleinMarg}.

\begin{lemma}\label{coc}
Fix $m, n \in \N$ and $x_{0} > 0$, and let $\psi : [x_{0},\infty) \to (0,\infty)$ be a non-increasing continuous function. Then there exists a unique continuous function $r : [a_{0},\infty) \to \bbR$ where $a_{0} = \frac{m}{m+n}\log x_{0} - \frac{n}{m+n}\log(\psi(x_{0}))$ such that

\begin{equation}\label{coc1}
\lambda(a) = a - nr(a)~\text{is strictly increasing and tends to}~\infty~\text{as}~a \to \infty,
\end{equation}

\begin{equation}\label{coc2}
L(a) = a + m r(a)~\text{is non decreasing}
\end{equation}

\noindent and

\begin{equation}\label{coc3}
\psi(s^{a - n r(a)}) = s^{-a-m r(a)}~\text{for all}~a \geq a_{0}.
\end{equation}

\noindent Conversely, given $a_{0} \in \bbR$ and a continuous function $r : [a_{0}, \infty) \to \bbR$ such that (\ref{coc1}, \ref{coc2}) hold, there exists a unique continuous non-increasing function 
$\psi : [x_{0},\infty) \to (0,\infty)$ with $x_{0} = s^{a_{0} - n r(a_{0})}$ satisfying (\ref{coc3}). Further, for a positive integer $q$,

\begin{equation}\label{coc4}
\int_{x_{0}}^{\infty}(\log x)^{q}\psi(x)dx < \infty \iff \int_{a_{0}}^{\infty}a^{q}s^{-(m+n)r(a)}da < \infty.
\end{equation}
\end{lemma}

\noindent Given $t \in \Z$, we define the diagonal matrix $g_t \in \SL_{m + n}(k)$ as follows

\begin{equation}\label{defg1}
g_{t} := \diag(\underbrace{X^{nt}, X^{nt},\dots, X^{nt}}_\text{m terms}, \underbrace{X^{-mt}, X^{-mt}, \dots, X^{-mt}}_\text{n terms})
\end{equation}

\noindent The following is one half of Theorem $8.5$ from \cite{KleinMarg}. We reproduce the argument. It should be noted that the other half of the lemma is used for the convergence case of Khintchine's theorem, which we are not addressing since it can be proved by elementary means. It should be possible to prove an analogue of this as well, but is not as immediate and needs to be phrased more precisely. See \cite{Ghosh} Theorem 2.5 for an example. 

\begin{lemma}\label{chofcoord}
Suppose there exist infinitely many $t \in \mathbb{Z}_{+}$ such that 
\begin{equation}\label{chofcoord1}
\Delta(g_{t}\Lambda) \geq r(t).
\end{equation}
\noindent Then $\Lambda$ is $(\psi, n)$-approximable. 
\end{lemma}

\begin{proof}

\noindent We first dispense with the case when
\begin{equation}\label{zero}
\bv^{(m)} = 0~\text{for some}~\bv.
\end{equation}

\noindent Then integral multiples of $\bv$ produce infinitely many solutions of (\ref{latt-approx}) and lattices with (\ref{zero}) are $(\psi, n)$-approximable for any $\psi$. We may therefore assume that (\ref{chofcoord1}) holds for arbitrarily large $t \in \Z_{+}$ so that
$$ \|\bv^{(m)}\|^{m} \leq  s^{-n - r(t)} = \psi(s^{t - nr(t)}) \leq \psi(\|\bv_{(n)}\|^n)$$
\noindent which shows the desired result.
\end{proof}

\noindent In order to prove Theorem \ref{KG-gen}, it remains to show that $\mathcal{B}(\Delta)$ is Borel-Cantelli for $g_1$. This follows from Theorem \ref{km1.8} and Theorem \ref{theorem-volumes}.

Similarly we can dynamically interpret the multiplicative problem in the following manner. For $\bt$ in the set
$$\mathfrak{d} := \{\bt = (t_1,\dots, t_r) \in \Z^{r}~:~\sum_{i = 1}^{r}t_i = 0\},$$

\noindent we set $\|\bt\|_{-} := \max\{|t_i|~:~t_i \leq 0\}$ and
\begin{equation}
g_{\bt} := \diag(X^{t_1}, \dots, X^{t_r}).
\end{equation}

\noindent The following is one half of Theorem 9.2 in \cite{KleinMarg}. The proof is identical to loc.cit. and very similar to the argument for Lemma \ref{chofcoord} so we do not reproduce it.

\begin{lemma}
Suppose there exist infinitely many $\bt \in \mathfrak{d}$ such that
$$ \Delta(\Lambda) \geq r(\|\bt\|_{-}) $$
\noindent then $\Lambda$ is $\psi$ multiplicatively approximable.
\end{lemma}

\noindent\textbf{Proof of Theorem~\ref{Skriganov}:}

In order to prove this Theorem, we would like to use Theorem \ref{km1.8}. The function $\bt \to \|\bt\|_{-}$ is a norm when restricted to a Weyl chamber of $\mathfrak{d}$. We decompose $\mathfrak{d}$ as a union of such chambers $\mathfrak{d}_j$ and note that $\{g_{\bt}~:~\bt \in \mathfrak{d}_j\}$ is clearly an ED sequence and $\mathcal{B}(\Delta)$ is Borel-Cantelli for $g_{\bt}$ where $\bt$ runs through the intersection of $\mathfrak{d}_j$ with an arbitrary lattice in $\mathfrak{d}$. The result now follows from Corollary 2.4 and Lemma 2.8 in \cite{KleinMarg}.

\section{Geodesic excursions}\label{section-geodesic}
\subsection{Logarithm Laws}
 Let $X$ denote the Bruhat-Tits building of $G$ and let $Y := X/\Gamma$ be the corresponding finite-volume, non-compact quotient of $X$, and denote by $\pi: X \rightarrow Y$ the natural projection. Let $\partial X$ denote the geodesic boundary of $X$, which can be naturally identified with the spherical building of $G$. Given $x \in X$ and $\theta \in \partial X$, let $L = L_{x, \theta} = \{r_s(x, \theta)\}_{s \geq 0}$ denote geodesic ray based at $x$ in the direction $\theta$. Let $\{C_t(x, \theta)\}_{t \in \N}$ denote the sequence of chambers $L$ passes through. We are interested in the statistical behavior of the excursions away from compact sets in $Y$ of the projection $\pi(L)$. We will prove the analogue of Theorem $1.2$ from \cite{KleinMarg} in positive characteristic.
 We fix $y_0 \in Y$. Let $\mu$ be the natural $G$-invariant measure on $X$, normalized to be a probability measure on $Y$, which we also denote $\mu$. Define
\begin{equation}\label{defcusp}
A(r) := \{ y \in Y~:~d(y_0, y) \geq r\}
\end{equation}
\noindent and let
$$l(Y) := \lim_{r \to \infty} \frac{-\log(\mu(A(r)))}{r}.$$
\begin{theorem}\label{loglaw1}
Fix $y_0 \in Y$ and let $\{r_t ~:~t \in \mathbb{N}\}$ be an arbitrary sequence of integers. Then for any $x \in X$ and almost every (resp. almost no) \footnote{with respect to the natural measure class} $\theta \in \partial X$,
$$ d(\pi(C_t(x, \theta)), y_0) \geq r_t$$
provided the series
$$ \sum_{t = 1}^{\infty} s^{-l(Y)r_t}$$
\noindent diverges (resp. converges).
\end{theorem}

\noindent From this the positive characteristic analogue of the \emph{logarithm law for geodesics} follows. 

\begin{corollary}\label{loglawp} For any $x \in X$, $y_0 \in Y$ and almost all (with respect to the natural measure class) $\theta \in \partial X$, $$\limsup_{t \rightarrow \infty} \frac{ \log d(\pi(C_t(x, \theta)), y_0)}{ \log t} = \frac{1}{l(Y)}.$$ 
\end{corollary}\medskip

\subsection{A brief literature survey of logarithm laws}

The terminology `logarithm law' was coined (to our knowledge) by Sullivan~\cite{Sullivan}, who studied the phenomenon of excursions of geodesics in non-compact finite-volume hyperbolic manifolds. His results were extended to geodesic flows on symmetric spaces by Kleinbock-Margulis~\cite{KleinMarg} who also described how a similar philosophy could be used to derive results in classical diophantine approximation. 
Their paper, and the series of beautiful papers of F. Paulin~\cite{Paulin}, S. Hersonsky (\cite{HP1, HP2, HP3, HP4}) and of F. Paulin and J. Parkonnen (\cite{PP1, PP2, PP3, PP4}) describing a connection between geodesics and horoballs in trees and diophantine approximation in positive characteristic, provided much of the inspiration for this paper. Note that $G$ acts by isometries on its Bruhat-Tits building $X$, but is typically not the full automorphism group. The latter (denoted $\Aut(X)$) is a locally compact group and it makes sense to talk about lattices $\Gamma$ in $\Aut(X)$, which in particular may not arise ``algebraically", i.e. from $G$, and the geodesic flow on (unit tangent bundles of) $X/\Gamma$. In \cite{HP1}, logarithm laws were proved in this more general setup in the case that $X$ is a locally finite tree and that $\Gamma$ is a geometrically finite lattice. This category includes all algebraic examples, and so also the rank $1$ examples we consider in this paper. However, we are able to generalize logarithm laws to higher rank groups.

In order to prove Theorem~\ref{loglawp}, we will show to `code' the sequence of chambers a geodesic passes through by the orbit of a fixed chamber under a sequence of group elements. This will allow us to apply Theorem~\ref{km1.8} to a carefully chosen sequence of functions on $G/\Gamma$. This section draws heavily from~\cite{Mozes}, \S2. 

\medskip

\noindent\textbf{Remarks:} \begin{itemize}

\item One small (unimportant) difference between loc. cit. and this one is the side on which we have the group act on the building. Here, our group will act on the right, and consider the coset space $G/\Gamma$.

\item A more important difference: in loc. cit. the quotient complex $Y = X/\Gamma$ was a finite complex, since $\Gamma$ was a uniform lattice. Here, our quotients will be countable complexes. However this does not affect the approach.

\end{itemize}

\medskip

We recall that $X$ is the (geometric realization) of the Bruhat-Tits~\cite{BruTit} building of $G$, $Y = X/\Gamma$ the (finite-volume, non-compact) quotient, and $\pi: X \rightarrow Y$ the natural projection. $Y$ is a \emph{countable} complex, which inherits a labeling from the labeling on $X$. Following loc. cit. \S1, we say that a map between two labelled complexes is called \emph{admissible} if it is a label-preserving local homeomorphism.

Fix an apartment $\A$ and a chamber $C \in \A$. Let $B$ be the stabilizer of $C$ in $G$, $N$ the stabilizer of $\A$, and $T \subset N$ the subgroup that fixes $\A$ pointwise. Finally, let $A$ denote the subgroup of $N$ which acts via translations on $\A$. Let
$$\Omega : = \{\omega: \A \rightarrow Y|\omega \mbox{ is an admissible map }\}$$

\noindent One can think of $\Omega$ as the number of ways we can `fold' $\A$ into $Y$. We define $\phi: G/\Gamma \rightarrow \Omega$ via 
$$\phi(g\Gamma) = \pi \circ g|_{\A}$$

\noindent  This map is well-defined, and by Proposition 2.3 of loc. cit. (mutatis mutandis), is onto.
Note that the group $A$ acts on $\Omega$ by $a \omega = \omega \circ a|_{\A}$, and $\phi$ intertwines this action with the natural action of $A$ on $G/\Gamma$. Note that the $A$ action on $G/\Gamma$ factors through the action of $T$. Endowing $\Omega$ with the measure $\nu$ induced by Haar measure on $G/\Gamma$, we have, via Corollary 2.8 of loc. cit., an isomorphism between the dynamical systems $(T\backslash G/\Gamma, A)$ and $(\Omega, A)$.

The main step in the proof of Theorem~\ref{loglawp} is the following:

\begin{theorem}\label{chambers} Fix $y_0 \in Y$. Let $L \in \A$ be a geodesic ray (that is, a straight euclidean line segment). Let $\{C_t\}_{t \in \N}$ be the sequence of chambers it passes through. Let $\{r_t\}_{t \in \N}$ denote a sequence of real numbers. Then 

\begin{equation}
\nu(\omega: d_Y(\omega(C_t), y_0) \geq r_t \mbox { infinitely often }) = \left\{
\begin{array}{ll}  1 & \mbox{if}~ \sum_{t \in \N} s^{-l r_t} \mbox{ diverges }\\\\ 0 & \mbox{otherwise}.\end{array}\right.\end{equation}

\end{theorem}

\medskip

\noindent We first indicate how to derive Theorem~\ref{loglawp} from Theorem \ref{chambers}:

\medskip

\noindent\textbf{Proof of Theorem~\ref{loglawp}:} Recall that $\partial X$ denotes the geodesic boundary of $X$, which can be naturally identified with the spherical building of $G$. The set of chambers of $\partial X$ is given by $B_{\infty}\backslash G$, where $B_{\infty}$ is the stabilizer of a chamber in $\partial X$. We endow $\partial X$ with a measure class by considering the product of Haar measure on $B_{\infty}\backslash G$ with Lebesgue measure on apartments. 

Given $x \in X, \theta \in \partial X$, let $L = L_{x, \theta} = \{r_s(x, \theta)\}_{s \geq 0}$ denote geodesic ray based at $x$ in the direction $\theta$. Let $\A^{\prime}$ be the apartment containing $L$, and $\{C_t(x, \theta)\}_{t \in \N}$ denote the sequence of chambers $L$ passes through. Let $\epsilon >0$, and suppose there were a positive measure set $\mathbf{B}_{\epsilon}$ of $\theta \in \partial X$ with 

\begin{equation}\label{loglawlower} \limsup_{t \rightarrow \infty} \frac{ \log d_Y (\pi(C_t(x, \theta)), y_0)}{ \log t} \geq \frac{1}{l} + \epsilon \end{equation}

\noindent Note that $\eta(\theta) =  \limsup_{t \rightarrow \infty} \frac{ \log d_Y (\pi(C_t(x, \theta)), y_0)}{ \log t}$ is clearly a $\Gamma$-invariant function of $\theta \in \partial X$. Thus, there is a positive $\Gamma$-invariant (Haar) measure set of chambers $E$ in $B_{\infty}\backslash G$ such that in each there is a positive Lebesgue measure set of geodesic rays satisfying (\ref{loglawlower}). We will use lowercase letters to denote chambers in $\partial X$. We write 
$$\mathbf{B}_{\epsilon} = \bigcup_{c \in E} \mathbf{B}_{c, \epsilon}$$

\noindent Note that the Lebesgue measure of the set $\mathbf{B}_{c, \epsilon}$ is a $\Gamma$-invariant function on $B_{\infty}\backslash G$, and so constant. Thus, we see that $E$ must contain almost every point in $B_{\infty}\backslash G$.  Using strong transitivity of the $G$-action (on $\partial X$), we pull these sets $\mathbf{B}_{c, \epsilon}$ back to the apartment $\mathfrak{a}_0$ sitting at the boundary of our fixed apartment $\A$. Let $\mathfrak{a}$ be the apartment containing $c$, and let $g$ be so that $\mathfrak{a}_0 g = \mathfrak{a}$. Then $\mathbf{B}_{\mathfrak{a}, \epsilon}g^{-1} \subset \mathfrak{a}_0$.

Fix a geodesic ray $L \subset \A$ in one of these directions, and let $\{C_t\}_{t \in \N}$ be the sequence of chambers it passes through.  Lifting the set $E$ to $G$ and projecting to $G/\Gamma$, and then to $\Omega$, we obtain a positive measure set of admissible maps $\omega$ for which 

$$d_Y(\omega(C_t), y_0) \geq (1+\epsilon) \frac{1}{l} \log t \mbox { infinitely often } $$

\noindent which contradicts Theorem~\ref{chambers}. A similar argument shows that the set of $\theta \in \partial_X$ with 

\begin{equation}\label{loglawupper} \limsup_{t \rightarrow \infty} \frac{ \log d_Y (\pi(C_t(x, \theta)), y_0)}{ \log t} \le \frac{1}{l} - \epsilon \end{equation}

\noindent has measure $0$ for all $\epsilon >0$. \qed\medskip

\noindent\textbf{Proof of Proposition~\ref{chambers}:} Given a geodesic ray $L$ and the associated sequence of chambers $\{C_t\}$, let $\{a_t\}$ denote a sequence of elements in $A$ so that $C_0 a_t = C_t$. We fix $y_0 \in Y$ and let $x_0 \in p^{-1}(y_0) \in G/\Gamma$. Now define 

$$h_t := \{x \in G/\Gamma~:~d(a_t x, x_0) > r_t\}$$
\noindent We apply Theorem \ref{km1.8} with the family $\mathcal{B} = \{h_t\}$, and the sequence of group elements $\{a_t\}$. By Theorem \ref{theorem-volumes} we know that $d(x_0,~)$ is $\DL$, moreover $\{a_t\}$ clearly satisfy the exponential divergence condition (\ref{ED}), since $d(a_t, e)$ (where $e$ is the identity element) grows linearly in $t$. 
\qed\medskip

\subsection{Flats} A similar argument as above also yields a logarithm law for \emph{flats} in the building. Given a $d$-dimensional flat ($1 \le d \le r$, where $r = \rank(\G)$), i.e. an isometrically embedded $d$-dimensional Euclidean space in $X$, we can consider the behavior of its  projection to $Y$. That, is, let $\mathfrak{d}: \R^d \rightarrow X$ denote the flat. Given $\bf{t} \in \R^d$, let $C_{\mathbf{t}, n}$ denote the chamber containing $\mathfrak{d}(n\mathbf{t})$. One considers the sequences of chambers obtained this way and applies a suitably modified version of Proposition~\ref{chambers}. For a $d$-dimensional flat, one will obtain a generic limit of $d/l$. This is similar to the argument used to prove Theorem 1.10 of~\cite{KleinMarg}.

\section{Concluding remarks}\label{section-concluding}

\subsection{Hausdorff measure and dimension}
In this section, we describe briefly how our theorems, coupled with the Mass Transference Principle of Beresnevich-Velani \cite{BerVel} can be used to obtain information about the Hausdorff measure and dimension of the sets we have been studying. A dimension function $f : \bbR_{+} \to \bbR_{+}$ is a continuous, nondecreasing function such that $\lim_{x \to 0}f(x) = 0$. Given such a function and a ball $B$ of radius $\theta$ in a metric space $(X, d)$, its $f-\volume$ is defined by
\begin{equation}\label{f-vol-def}\nonumber
V^{f}(B) = f(\theta).
\end{equation}

\noindent The Hausdorff $f$-measure $\cH^{f}(E)$ of a subset $E$ of $X$ with respect to the dimension function $f$ is defined as
\begin{equation}\label{def-H-measure}\nonumber
\cH^{f} (E) = \lim_{\rho \to 0} \cH^{f}_{\rho}(E) 
\end{equation}

\noindent where $\cH^{f}_{\rho}(E)$ is defined by

$$\cH^{f}_{\rho}(E) = \inf \sum_{i} V^{f} (B_i) $$

\noindent and $\{B_i\}$ forms a countable cover of $E$, each $B_i$ has radius at most $\rho$ and the infimum is taken over all such countable covers. A typical example of a dimension function is $f(\theta) = \theta^s$ for nonnegative $s$, in this case $\cH^f$ is called the $s$ dimensional Hausdorff measure and denoted by $\cH^{s}$. The Hausdorff dimension of a subset $E$ is defined by
$$ \dim E = \inf \{s~:~\cH^{s}(E) = 0 \} = \sup \{s~:~ \cH^{s}(E) = \infty\}.$$

\noindent  Beresnevich and Velani proved a Mass Transference Principle (\cite{BerVel}, Theorems 2 and 3) which allows us to derive Hausdorff measure bounds from  the ones for Haar measure. In short, many \emph{divergence} type Borel-Cantelli statements can be recast in terms of Hausdorff measure using their important technique. In particular, all such theorems in this paper are subject to this technology.  We only provide one example \footnote{The Hausdorff measure and dimension of \emph{exceptional sets} is of considerable interest both in dynamical as well as number theoretic contexts. A more complete treatment of this subject will be given elsewhere.  See also \cite{BDV} for more applications of \cite{BerVel} and related results.}: a direct application of Mass Transference coupled with Theorem \ref{KG-gen} gives us

\begin{theorem}\label{H-measure}
Let $f$ be a dimension function such that $x^{-n}f(x)$ is a decreasing function which tends to $\infty$ as $x$ tends to $0$. Assume that
\begin{equation}\label{H-sum-div}
\sum_{r = 1}^{\infty} f(\psi(r))\psi(r)^{-(n-1)} = \infty.
\end{equation}
\noindent Then
$$\cH^{f}(\{\Lambda \in \SL_{n}(k)/\SL_{n}(Z)~:~\Lambda~\text{is}~(\psi, n)~\text{approximable}\}) = \infty.$$ 
\end{theorem}

\end{document}